\documentclass{amsart}

\usepackage{amssymb,amsmath,amsthm,amsfonts}
\usepackage{graphicx}
\usepackage{amsthm}
\usepackage{amsmath}
\usepackage{color}
\usepackage{mathabx}
\usepackage[all,cmtip]{xy}
\usepackage{tikz-cd}
\usepackage{stmaryrd}
\usepackage{hyperref}
\usepackage{algorithm}
\usepackage{algpseudocode}
\usepackage{pifont}
\usepackage[margin=1in]{geometry}
\usepackage{mathtools}
\usepackage[percent]{overpic}

\newtheorem{thm}{Theorem}[section]
\newtheorem*{thm*}{Theorem}
\newtheorem{prop}[thm]{Proposition}
\newtheorem{lem}[thm]{Lemma}

\newtheorem*{maintheorem}{Main Theorem}

\newcommand{\R}{\mathbb{R}}

\newcommand{\C}{\mathbb{C}}

\newcommand{\Gr}{\mathrm{Gr}_k(\C^N)}
\newcommand{\St}{\mathrm{St}_k(\C^N)}
\newcommand{\Fl}{\mathrm{Fl}}

\begin{document}


\title{Symplectic Geometry and Connectivity of Spaces of Frames}

\author{Tom Needham and Clayton Shonkwiler}

\maketitle

\begin{abstract}
	Frames provide redundant, stable representations of data which have important applications in signal processing. We introduce a connection between symplectic geometry and frame theory and show that many important classes of frames have natural symplectic descriptions. Symplectic tools seem well-adapted to addressing a number of important questions about frames; in this paper we focus on the frame homotopy conjecture posed in 2002 and recently proved by Cahill, Mixon, and Strawn, which says that the space of finite unit norm tight frames is connected. We give a simple symplectic proof of a double generalization of the frame homotopy conjecture, showing that spaces of complex frames with arbitrary prescribed norms and frame operators are connected. To spark further investigation, we also suggest a number of fundamental questions in frame theory which seem amenable to a symplectic approach.
\end{abstract}

\section{Introduction}
Speaking loosely, a frame in a Hilbert space $\mathcal{H}$ is an overcomplete basis for $\mathcal{H}$. The overcompleteness of a frame allows for both greater flexibility and greater robustness to data loss, both of which are of substantial importance in a variety of applications.

Frames have a long history in the signal processing community, having been introduced by Duffin and Schaeffer in 1952~\cite{Duffin:1952kh}, though they were relatively neglected until Daubechies, Grossmann, and Meyer's pioneering work on wavelets in the 1980s~\cite{Daubechies:1986kk}. In the 21st century, an interest in \emph{finite frames} (when $\mathcal{H} = \R^N$ or $\C^N$) led to an explosion of theoretical work and new applications; see~\cite{Kovacevic:2007fja,Kovacevic:2007dba,Casazza:2013bn} for an introduction.

While much of the modern work on finite frames leverages the algebraic geometry of frame varieties~\cite{Cahill:2013fy}, we know of no previous attempts to use symplectic geometry to study frames: the goal of this paper is to show that spaces of complex frames are closely related to nice symplectic manifolds, where some very powerful results from symplectic geometry apply. As an application, we prove a double generalization of the \emph{frame homotopy conjecture} -- which says that the space of unit-norm tight frames is path-connected -- but this application of the symplectic machinery only scratches the surface of a potentially fruitful connection between frame theory and symplectic geometry, so we hope that this paper will inspire others to explore this connection further. 

In order to describe the setting and our results, we recall some definitions. Our results are solely about frames in finite-dimensional complex vector spaces, so we state our definitions in that setting.

A \emph{frame} in $\C^k$ is a collection $F=\{f_j\}_{j=1}^N$ of vectors $f_j \in \C^k$ satisfying 
\begin{equation}\label{eq:tight}
a \|v \|^2 \leq \sum_{j=1}^N |\left<v,f_j\right>|^2 \leq b \|v\|^2 \quad \forall \; v \in \C^k
\end{equation}
for some numbers $0 < a \leq b$ called \emph{frame bounds}. Throughout the paper, we use $\left<\cdot,\cdot\right>$ to denote the standard Hermitian product on $\C^k$ for any $k$ and $\|\cdot\|$ its induced norm. Let $\mathcal{F}^{N,k}$ denote the space of frames of $N$ vectors in $\C^k$. Since the parameters $N$ and $k$ will for the most part be fixed throughout the paper, we typically shorten our notation to $\mathcal{F}=\mathcal{F}^{N,k}$. Identifying a frame $F=\{f_j\}_{j=1}^N \in \mathcal{F}$ with the $k \times N$ matrix with columns given by the vectors $f_j$ represented in the standard basis, $\mathcal{F}$ can be viewed as an open, dense subset of $\C^{k \times N}$.

When $a=b$ in \eqref{eq:tight}, the frame satisfies a scaled Parseval identity $\sum_{j=1}^N |\left<v,f_j\right>|^2 = a\|v\|^2$, and hence such frames are particularly useful in signal reconstruction problems and are called $a$-\emph{tight} (or just \emph{tight}). Tight frames are natural from the symplectic perspective, though this is most easily seen through the following alternative characterization: any frame $F$ has an associated \emph{analysis operator}
\begin{align} \label{eqn:analysis}
T_F:\C^k &\rightarrow \C^N \\
\nonumber v &\mapsto (\left<v,f_1\right>,\ldots,\left<v,f_N\right>),
\end{align}
a \emph{synthesis operator}
\begin{align*}
T_F^\ast:\C^N &\rightarrow \C^k \\
(z_1,\ldots,z_N) &\mapsto \sum_{j=1}^N z_j f_j,
\end{align*}
and a \emph{frame operator}
$$
S_F=T_F^\ast T_F: \C^k \rightarrow \C^k.
$$
Again thinking of the frame $F$ as a $k \times N$ matrix, the above operators are expressed in terms of matrix multiplication as $T_F(v) = F^\ast v$ and $T_F^\ast(v)= F v$, and the frame operator is therefore given by $S_F = F F^\ast$. It is straightforward to show that a frame $F$ is $a$-tight if and only if its frame operator satisfies $S_F=\frac{1}{a} \mathrm{Id}_k$, where $\mathrm{Id}_k$ denotes the identity map on $\C^k$. We use the notation $\mathcal{F}_S=\mathcal{F}^{N,k}_S$ to indicate the space of frames with prescribed frame operator $S: \C^k \to \C^k$; in particular, $\mathcal{F}_{\frac{1}{a}\mathrm{Id}_k}$ is the space of $a$-tight frames. In symplectic terminology (which we will define precisely below), the map $\mu_{\mathrm{U}(k)}:F \to F F^\ast$ is the momentum map of the (left) Hamiltonian $\mathrm{U}(k)$ action on $\C^{k \times N}$, so each space $\mathcal{F}_S$ -- including the $a$-tight frame space $\mathcal{F}_{\frac{1}{a}\mathrm{Id}_k}$ -- is a level set of this map. 

Another interesting class of frames are the \emph{unit-norm frames}, which have all $\|f_j\|^2 = 1$. In a signal reconstruction context these frames produce measurements of equal statistical power. The space of all unit-norm frames also has a natural symplectic description: as we will see, the map $\mu_{\mathrm{U}(1)^N}: \C^{k \times N} \to \R^N$ defined by $\mu_{\mathrm{U}(1)^N}(F) = \left(-\frac{1}{2}\|f_1\|^2, \dots , -\frac{1}{2} \|f_N\|^2\right)$ is the momentum map of the (right) Hamiltonian action of the torus  of diagonal, unitary $N \times N$ matrices on $\C^{k \times N}$. For each $\vec{r} = (r_1, \dots , r_N) \in \R^N$ with $r_j \geq 0$, the space $\mathcal{F}(\vec{r})$ of frames with $\|f_j\|^2 = r_j$ is a level set of this momentum map. For example, the space of unit-norm frames $\mathcal{F}(1,\dots , 1) = \mu_{\mathrm{U}(1)^N}^{-1}\left(-\frac{1}{2},\dots , -\frac{1}{2}\right)$.

The frames which are both tight and unit-norm are called \emph{finite unit-norm tight frames} (or FUNTFs). The interest in FUNTFs is due in part to the fact that they are optimal for signal reconstruction when each measurement has equal power in the presence of additive white Gaussian noise and erasures~\cite{Goyal:2001cd,Casazza:2003vp,Holmes:2004iv}. For example, Rupf and Massey~\cite{Rupf:1994fl} showed that, when all users have the same power, optimal signature sequences in CDMA correspond to FUNTFs.

A FUNTF $F$ must have frame bound $a=\frac{k}{N}$ since $FF^\ast = \frac{1}{a} \mathrm{Id}_k$ and $\operatorname{trace}(F F^\ast) = \operatorname{trace}(F^\ast F) = \sum_{i=1}^N \|f_i\|^2 = N$, meaning that the space of FUNTFs is exactly $\mathcal{F}_{\frac{N}{k}\mathrm{Id}_k} \cap \mathcal{F}(\vec{1})$, where $\vec{1}$ is the vector of all $1$s. Therefore, FUNTF space is the intersection
\[
	\mu_{\mathrm{U}(k)}^{-1}\left(\frac{N}{k}\mathrm{Id}_k\right) \cap \mu_{\mathrm{U}(1)^N}^{-1}\left(-\frac{1}{2},\dots , -\frac{1}{2}\right),
\]
or, equivalently, $\mu^{-1}\left(\frac{N}{k}\mathrm{Id}_k,\left(-\frac{1}{2},\dots , -\frac{1}{2}\right) \right)$, where $\mu$ is the momentum map of an induced Hamiltonian $\mathrm{U}(k) \times \mathrm{U}(1)^N$ action on $\C^{k \times N}$. More generally, each level set of $\mu$ corresponds to a space $\mathcal{F}_S(\vec{r})$ of frames with prescribed frame operator $S$ and prescribed squared frame norms $\|f_j\|^2=r_j \geq 0$.

The main message of this paper is that this symplectic perspective provides insight into the structure of these spaces, which are generally poorly understood. Even the FUNTF spaces remain fairly mysterious: they are known to be non-empty~\cite{Goyal:2001cd,Zimmermann:2001kp}, but the simplest possible question about the topology of FUNTF spaces -- namely, \emph{are they path-connected?} -- is known as the \emph{frame homotopy conjecture}. This problem was posed by Larson in a 2002 REU and it first appears in the literature in Dykema and Strawn's 2006 paper~\cite{dykema2003manifold}, in which Conjecture~7.7 states ``the space of unit norm, tight complex $N$ frames in $\C^k$ is connected for all $N,k$ with $N \geq k > 1$." The frame homotopy conjecture was proved for the particular case $N=2k$ in \cite{giol2009projections}, but the full conjecture remained open until 2017, when it was proved by Cahill, Mixon, and Strawn~\cite{cahill2017connectivity}. 

The existing proof of the frame homotopy conjecture is rather technical and, while in principle it might be generalizable to other $\mathcal{F}_S(\vec{r})$ spaces, in practice that would be a fairly daunting challenge. However, we will see that the frame homotopy conjecture follows rather easily from some classical theorems from symplectic geometry, and moreover that this strategy applies just as well to all other choices of frame operator and squared frame vector norms:

\begin{maintheorem}
For any frame operator $S$ and for any admissible vector of squared norms $\vec{r}$, the space $\mathcal{F}_S(\vec{r})$ is path-connected.
\end{maintheorem}

As mentioned above, FUNTFs provide optimal reconstructions in the context of measurements of equal power with additive white Gaussian noise. When measurements have unequal power, however, FUNTFs are not optimal. For example, Viswanath and Anantharam showed that optimal CDMA signature sequences when users have different powers correspond to tight frames with squared norms of the frame vectors proportional to user powers~\cite{Viswanath:1999hf,Casazza:2006cx}. Moreover, quoting Casazza et al.~\cite{Casazza:2011ev}, in the presence of colored noise ``a tight frame is no longer optimal and the frame operator needs to be matched to the noise covariance matrix'' (cf.~\cite{Bolcskei:2001et,Viswanath:2002bv}). Therefore, the spaces $\mathcal{F}_S(\vec{r})$ of frames with prescribed frame operator and squared frame vector norms provide optimal reconstructions in the context of inhomogeneous measurement power and/or colored noise.

We review the relevant ideas from symplectic geometry in Section~\ref{sec:symplectic geometry}, and then, to introduce these ideas in a more familiar setting, we recover the result of Cahill, Mixon and Strawn for complex FUNTFs in Section~\ref{sec:FUNTFs}. We prove the main theorem in Section~\ref{sec:general_case} and suggest some other questions in frame theory which seem accessible to symplectic techniques in Section~\ref{sec:discussion}.

\section{Basic Concepts from Symplectic Geometry} \label{sec:symplectic geometry}

A good reference for the definitions and results presented here is \cite{mcduff2017introduction}. A \emph{symplectic manifold} is a smooth, even-dimensional manifold $M$ equipped with a closed, nondegenerate $2$-form $\omega$. Let $G$ denote a Lie group which acts on $M$ and let $\mathfrak{g}$ denote its Lie algebra. Each point $\xi \in \mathfrak{g}$ determines an \emph{infinitesimal vector field} $X_\xi$ by the formula 
$$
X_\xi(p) = \left.\frac{d}{d\epsilon}\right|_{\epsilon=0} \mathrm{exp}(\epsilon \xi) \cdot p
$$
for each $p \in M$. In this expression, $\exp:\mathfrak{g} \rightarrow G$ is the exponential map, so that the quantity being differentiated on the righthand side represents the action of a Lie group element on the point $p$ for each value of $\epsilon$. Using $\mathfrak{g}^\ast$ to denote the dual of the Lie algebra $\mathfrak{g}$, a \emph{momentum map} for the action is a smooth map $\mu_G:M \rightarrow \mathfrak{g}^\ast$ satisfying 
$$
\omega_p(X,X_\xi(p)) = D_p\mu_G(X)(\xi).
$$
for each $p \in M$, $\xi \in \mathfrak{g}$ and $X \in T_p M$. The expression on the righthand side denotes the evaluation of $D_p \mu_G(X) \in T_{\mu_G(p)} \mathfrak{g}^\ast \approx \mathfrak{g}^\ast$ on the vector $\xi$. We note that some authors reverse the arguments of $\omega$ in this definition, so that our definition will differ by a sign due to skew-symmetry of $\omega$. When $G$ is abelian and $M$ is the phase space of a mechanical system, the momentum map simply records the conserved quantities guaranteed by Noether's theorem.

If the action of $G$ admits a momentum map, then the action is called \emph{Hamiltonian}. Hamiltonian actions play a special role in symplectic geometry and one important property is that they induce a quotient operation in the symplectic category. The quotient operation is referred to as a \emph{symplectic reduction} or \emph{Marsden--Weinstein--Meyer quotient} and is defined in the following classical theorem.

\begin{thm}[Marsden--Weinstein--Meyer Theorem \cite{marsden1974reduction,Meyer:1973wu}]\label{thm:marsden_weinstein}
Let $(M,\omega)$ be a symplectic manifold with a Hamiltonian action of a Lie group $G$ and let $\mu_G:M \rightarrow \mathfrak{g}^\ast$ denote the momentum map for this action. For any regular value $\xi \in \mathfrak{g}^\ast$ of $\mu_G$ which is fixed by the coadjoint action of $G$ and so that $G$ acts freely on $\mu_G^{-1}(\xi)$, the space
$$
M\sslash_\xi G := \mu_G^{-1}(\xi)/G
$$
has a natural symplectic structure $\widetilde{\omega}$ satisfying 
\begin{equation}\label{eqn:marsden-weinstein_form}
\iota^\ast \omega = \pi^\ast \widetilde{\omega},
\end{equation}
where $\iota:\mu_G^{-1}(\xi) \rightarrow M$ and $\pi:\mu_G^{-1}(\xi)\rightarrow \mu_G^{-1}(\xi)/G$ denote the inclusion and projection maps, respectively.

More generally, let $\xi \in \mathfrak{g}^\ast$ be an arbitrary regular value of $\mu_G$ and let $\mathcal{O}_\xi$ denote its coadjoint orbit. Then the space
$$
M\sslash_\xi G := \mu_G^{-1}(\mathcal{O}_\xi)/G
$$
has a natural symplectic structure satisfying the analogue of \eqref{eqn:marsden-weinstein_form}.
\end{thm}

Our main technical tool is the following theorem of Atiyah, which is part of the famous Atiyah--Guillemin--Sternberg convexity theorem. The original theorem is \cite[Theorem 1]{atiyah1982convexity} and the statement given here is  \cite[Theorem 5.21]{da2006symplectic}.

\begin{thm}[Atiyah's Connected Level Set Theorem]\label{thm:atiyah}
Let $(M,\omega)$ be a compact connected symplectic manifold with a Hamiltonian $n$-torus action with momentum map $\mu:M \rightarrow \R^n$. Then the nonempty level sets of $\mu$ are connected.
\end{thm} 

Throughout the paper we identify $(\R^n)^\ast \approx \R^n$, specifically  choosing the isomorphism to be the one induced by the standard inner product. 

\section{Finite Unit Norm Tight Frames}\label{sec:FUNTFs}

In this section we will identify the space of FUNTFs as the level set $\mu^{-1}\left(\frac{N}{k}\mathrm{Id}_k,\left(-\frac{1}{2},\dots , -\frac{1}{2}\right) \right)$ of the momentum map $\mu$ of a Hamiltonian action of $\mathrm{U}(k) \times \mathrm{U}(1)^N$ on the vector space $\C^{k \times N}$ of $k \times N$ complex matrices. In fact, in order to apply Theorem~\ref{thm:marsden_weinstein}, we will work with a subgroup $\mathrm{U}(k) \times \mathrm{U}(1)^{N-1}$ and, by slight abuse of notation, we will continue to use $\mu$ to denote the momentum map of this subgroup.

Since $\mathrm{U}(k) \times \mathrm{U}(1)^{N-1}$ is connected, showing that FUNTF space is connected is equivalent to showing that the quotient
\[
	\mu^{-1}\left(\frac{N}{k}\mathrm{Id}_k,\left(-\frac{1}{2},\dots , -\frac{1}{2}\right) \right)/\left(\mathrm{U}(k) \times \mathrm{U}(1)^{N-1}\right) = \C^{k \times N} \sslash_{\left(\frac{N}{k}\mathrm{Id}_k,\left(-\frac{1}{2},\dots , -\frac{1}{2}\right) \right)} \mathrm{U}(k) \times \mathrm{U}(1)^{N-1}
\]
is connected.

In turn, the strategy is to perform the reduction in stages:
\[
	\C^{k \times N} \sslash_{\left(\frac{N}{k}\mathrm{Id}_k,\left(-\frac{1}{2},\dots , -\frac{1}{2}\right) \right)} \mathrm{U}(k) \times \mathrm{U}(1)^{N-1} \approx \left(\C^{k \times N} \sslash_{\frac{N}{k}\mathrm{Id}_k} \mathrm{U}(k) \right)\sslash_{\left(-\frac{1}{2},\dots , -\frac{1}{2}\right)} \mathrm{U}(1)^{N-1}.
\]
As we will see, the inner symplectic quotient is the Grassmannian $\mathrm{Gr}_k(\C^N)$, which is well-known to be compact and connected, so Atiyah's theorem will imply that the level set of the momentum map for the $U(1)^{N-1}$ action on the Grassmannian is connected, and hence that the quotient is connected.

\subsection{The Symplectic Structure on $\C^{k \times N}$ and the Hamiltonian Actions}\label{sec:actions}
The matrix space $\C^{k \times N}$ has a symplectic structure associated to its standard Hermitian inner product, defined on $X_1,X_2 \in T_{F} \C^{k \times N} \approx \C^{k \times N}$ by
$$
\omega_V(X_1,X_2) = -\mathrm{Im}\left<X_1,X_2\right> = -\mathrm{Im} \; \mathrm{trace}(X_1^\ast X_2),
$$ 
where the product $\left<\cdot,\cdot\right>$ is interpreted as the Hermitian inner product on $\C^{k \cdot N}$, or equivalently the Frobenius inner product on $k \times N$ matrices. 

The group $\mathrm{U}(N)$ of $N \times N$ unitary matrices acts on $\C^{k \times N}$ by right multiplication, and hence so does the subgroup \(\mathrm{U}(1)^N\) of diagonal unitary matrices. The effect of this action on $F \in \C^{k \times N}$ is to independently change the phase of each column $f_j$ of $F$. The Lie algebra $\mathfrak{u}(1)^N \approx \R^N$ is the trivial $N$-dimensional Lie algebra, and we identify $(\mathfrak{u}(1))^\ast \approx \R^N$ via the isomorphism induced by the standard inner product.

\begin{prop}\label{prop:torus moment}
	The map
	\begin{align}\label{eqn:torus moment map}
		\nonumber \mu_{\mathrm{U}(1)^N}: \C^{k \times N} &\rightarrow \R^N  \\
		\left[ f_1 | f_2 | \cdots | f_N\right] & \mapsto \left(-\frac{1}{2}\|f_1\|^2 , -\frac{1}{2} \|f_2\|^2 , \dots , -\frac{1}{2}\|f_N\|^2\right)
	\end{align}
	is a momentum map for the $\mathrm{U}(1)^N$-action on $\C^{k \times N}$.
\end{prop} 

\begin{proof}
	Since the circle factors of the torus act independently on each column, it is enough to verify that the $\mathrm{U}(1)$ action on $\C^k$ given by 
	\[
		e^{it}\cdot f := f e^{it}
	\]
	has momentum map $\mu_{\mathrm{U}(1)}: f \mapsto -\frac{1}{2}\|f\|^2$. But this is clear: the vector field corresponding to $t \in \mathfrak{u}(1) \approx \R$ is
	\[
		X_t(f) = \left.\frac{d}{d\epsilon}\right|_{\epsilon=0} f e^{i \epsilon t} = i t f,
	\]
	and hence
	\[
		\omega_v(X, X_t(f)) = -\mathrm{Im}\langle X, i t f \rangle =-t\, \mathrm{Re} \langle X, f \rangle = -t\left. \frac{d}{d\epsilon}\right|_{\epsilon=0} \mathrm{Re}\langle f + \epsilon X, f + \epsilon X \rangle = D_f \mu_{\mathrm{U}(1)}(X)(t).
	\]
\end{proof}

In turn, the group $\mathrm{U}(k)$ acts on $\C^{k \times N}$ by left multiplication and this action is also Hamiltonian, as we will verify by identifying the momentum map. The Lie algebra $\mathfrak{u}(k)$ consists of the $k\times k$ skew-Hermitian matrices, and we identify the dual $\mathfrak{u}(k)^\ast$ with the space of Hermitian matrices $\mathcal{H}(k)$ via the map
\begin{align*}
\mathcal{H}(k) &\rightarrow \mathfrak{u}(k)^\ast \\
A &\mapsto \left(B \mapsto \frac{i}{2}\mathrm{trace}(AB)\right).
\end{align*}

\begin{prop}
The map
\begin{align*}
\mu_{\mathrm{U}(k)}: \C^{k \times N} &\rightarrow \mathcal{H}(k) \approx \mathfrak{u}(k)^\ast \\
F &\mapsto FF^\ast.\
\end{align*}
is a momentum map for the $\mathrm{U}(k)$-action on $\C^{k \times N}$.
\end{prop}

\begin{proof}
For $F \in \C^{k \times N}$, the derivative of $\mu_{\mathrm{U}(k)}$ is given by
\begin{align*}
D_F\mu_{\mathrm{U}(k)}:\C^{k \times N} &\rightarrow \mathcal{H}(k) \\
X &\mapsto  F X^\ast + X F^\ast.
\end{align*}
The infinitesimal vector field induced by $B \in \mathfrak{u}(k)$ is given by $X_B(F)=BF$ at each $F \in \C^{k \times N}$, and it follows that, for any vector $X \in T_F \C^{k \times N} \approx \C^{k \times N}$,
\begin{align}
\omega_F(X,X_B(F)) &= -\mathrm{Im}\; \mathrm{trace}\left( X^\ast BF\right) = \frac{i}{2} \mathrm{trace}\left(X^\ast BF - F^\ast B^\ast X\right) \nonumber \\
&= \frac{i}{2} \mathrm{trace}\left(FX^\ast B - X F^\ast B^\ast \right) = \frac{i}{2} \mathrm{trace}\left(FX^\ast B + X F^\ast B \right) \label{eqn:moment_map_1} \\
&= \frac{i}{2} \mathrm{trace}\left(D_F\mu_{\mathrm{U}(k)}(X) B\right) = D_F\mu_{\mathrm{U}(k)}(X)(B) \nonumber,
\end{align}
where the equalities in \eqref{eqn:moment_map_1} follow from the linearity and cyclic permutation-invariance of the trace operator and the assumption that $B$ is skew-Hermitian, respectively.
\end{proof}

The left $\mathrm{U}(k)$ and right $\mathrm{U}(1)^N$ actions commute since matrix multiplication is associative, so we can combine the above actions into a single Hamiltonian $\mathrm{U}(k) \times \mathrm{U}(1)^N$ action on $\C^{k \times N}$ with momentum map
\begin{align*}
	\mu: \C^{k \times N} & \to \mathcal{H}(k) \times \R^N \\
	F & \mapsto \left(FF^\ast, \left( -\frac{1}{2} \|f_1\|^2 , \dots , -\frac{1}{2}\|f_N\|^2\right) \right).
\end{align*}

Observe that the FUNFT space
\[
	\mathcal{F}_{\frac{N}{k}\mathrm{Id}_k}(\vec{1}) = \mu^{-1} \left(\frac{N}{k} \mathrm{Id}_k, \left(-\frac{1}{2},\dots , -\frac{1}{2}\right)\right),
\]
as desired.

However, some care is warranted: both $\mathrm{U}(k)$ and $\mathrm{U}(1)^N$ contain a circle subgroup of scalar matrices, and the actions of these two subgroups are redundant. To get a group which acts freely on the fiber -- which we need in order to apply Theorem~\ref{thm:marsden_weinstein} -- we can take the quotient
\[
	\left(\mathrm{U}(k) \times \mathrm{U}(1)^N\right)/U(1) \approx \mathrm{U}(k) \times \left(\mathrm{U}(1)^N/U(1)\right) \approx \mathrm{U}(k) \times \mathrm{U}(1)^{N-1},
\]
where an explicit isomorphism is given by thinking of elements of $U(1)^{N-1}$ as elements of $U(1)^N$ with the last diagonal entry fixed to be 1. The corresponding momentum map, which we also call $\mu$, is
\begin{align*}
	\mu: \C^{k \times N} & \to \mathcal{H}(k) \times \R^{N-1} \\
	F & \mapsto \left(FF^\ast, \left( -\frac{1}{2} \|f_1\|^2 , \dots , -\frac{1}{2}\|f_{N-1}\|^2\right)\right).
\end{align*}
Indeed, the redundancy of the full $\mathrm{U}(k) \times \mathrm{U}(1)^N$ action is revealed by the fact that
\[
	\sum_{i=1}^N -\frac{1}{2}\|f_j\|^2 = -\frac{1}{2} \operatorname{trace}(F^\ast F) = -\frac{1}{2}\operatorname{trace}(FF^\ast),
\]
so $\|f_N\|^2$ can be determined from the other $\|f_j\|^2$ when $FF^\ast$ is fixed .

It is not hard to see that the identity matrix $\mathrm{Id}_k \in \mathcal{H}(k)$ is a regular value of $\mu_{\mathrm{U}(k)}$ (we will characterize all regular values of the momentum map in Section \ref{sec:regular_values}), and likewise that $ \left( -\frac{1}{2}, \dots , -\frac{1}{2}\right) \in \R^{N-1}$ is a regular value of $\mu_{\mathrm{U}(1)^{N-1}}$. Hence, $\left(\frac{N}{k} \mathrm{Id}_k, \left(-\frac{1}{2},\dots , -\frac{1}{2}\right)\right)$ is a regular value of the product momentum map~$\mu$. Moreover, the coadjoint action of $\mathrm{U}(k)$ on $\mathcal{H}(k)$ is simply the conjugation action, which fixes $\frac{N}{k}\mathrm{Id}_k$, and the coadjoint action of $\mathrm{U}(1)^{N-1}$ on $\R^{N-1}$ is trival, so the product coadjoint action fixes the point $\left(\frac{N}{k} \mathrm{Id}_k, \left(-\frac{1}{2},\dots , -\frac{1}{2}\right)\right)$ and it is sensible to talk about taking the symplectic reduction over this point.

Since $\mathrm{U}(k) \times \mathrm{U}(1)^{N-1}$ is connected and FUNTF space is precisely $\mu^{-1} \left(\frac{N}{k} \mathrm{Id}_k, \left(-\frac{1}{2},\dots , -\frac{1}{2}\right)\right)$, the space of FUNTFs is connected if and only if the quotient
\[
	\mu^{-1} \left(\frac{N}{k} \mathrm{Id}_k, \left(-\frac{1}{2},\dots , -\frac{1}{2}\right)\right)/\left(\mathrm{U}(k) \times \mathrm{U}(1)^{N-1}\right) =: \C^{k \times N} \sslash_{\left(\frac{N}{k} \mathrm{Id}_k, \left(-\frac{1}{2},\dots , -\frac{1}{2}\right)\right)} \mathrm{U}(k) \times \mathrm{U}(1)^{N-1}
\]
is connected.

\subsection{Reduction in Stages and Grassmannians}

Rather than take the quotient all at once, it is convenient to perform reduction in stages:
\[
	\C^{k \times N} \sslash_{\left(\frac{N}{k}\mathrm{Id}_k,\left(-\frac{1}{2},\dots , -\frac{1}{2}\right) \right)} \mathrm{U}(k) \times \mathrm{U}(1)^{N-1} \approx \left(\C^{k \times N} \sslash_{\frac{N}{k}\mathrm{Id}_k} \mathrm{U}(k) \right)\sslash_{\left(-\frac{1}{2},\dots , -\frac{1}{2}\right)} \mathrm{U}(1)^{N-1}.
\]
However, the reduced space
\[
	\C^{k \times N} \sslash_{\frac{N}{k}\mathrm{Id}_k} \mathrm{U}(k) := \mu_{\mathrm{U}(k)}^{-1}\left(\frac{N}{k}\mathrm{Id}_k \right)/\mathrm{U}(k)
\]
is naturally identified with the Grassmannian $\Gr$ of $k$-dimensional linear subspaces of $\C^N$. To see this, observe that
\[
	\mu_{\mathrm{U}(k)}^{-1}\left(\frac{N}{k}\mathrm{Id}_k \right) = \left\{F \in \C^{k \times N} : F F^\ast = \frac{N}{k}\mathrm{Id}_k \right\}
\]
consists of all $k \times N$ matrices with orthogonal rows, each of norm $\sqrt{\frac{N}{k}}$. This is just a scaled copy of the Stiefel manifold $\St$ of $k$-tuples of Hermitian orthonormal vectors in $\C^N$, and the (free) left $\mathrm{U}(k)$ action corresponds to the standard action on the Stiefel manifold, meaning that the quotient is homeomorphic to $\Gr = \St/U(k)$. The Grassmannian is understood very well from a geometrical and topological perspective -- for example, it is a Riemannian symmetric space and a K\"{a}hler manifold -- and in particular it is known to be connected.

It follows from Theorem \ref{thm:marsden_weinstein} that the Grassmannian inherits a natural symplectic structure from $\C^{k \times N}$ which is compatible with the symplectic structure on $\C^{k \times N}$ in the sense of~\eqref{eqn:marsden-weinstein_form}. Moreover, the $\mathrm{U}(1)^{N-1}$ action commutes with the $\mathrm{U}(k)$ action and the momentum map~\eqref{eqn:torus moment map} is $\mathrm{U}(k)$-equivariant, so both the action and the momentum map descend to $\Gr$. The torus $\mathrm{U}(1)^{N-1}$ acts freely on the fiber over $\left(-\frac{1}{2},\dots , -\frac{1}{2}\right)$ and hence the symplectic reduction
\[
	\Gr \sslash_{\left(-\frac{1}{2},\dots , -\frac{1}{2}\right)} \mathrm{U}(1)^{N-1} \approx \C^{k \times N} \sslash_{\left(\frac{N}{k}\mathrm{Id}_k,\left(-\frac{1}{2},\dots , -\frac{1}{2}\right) \right)} \mathrm{U}(k) \times \mathrm{U}(1)^{N-1} 
\]
makes sense. 

Moreover, since the Grassmannian is connected and
\[
	\Gr \sslash_{\left(-\frac{1}{2},\dots , -\frac{1}{2}\right)} \mathrm{U}(1)^{N-1} = \mu_{\mathrm{U}(1)^{N-1}}^{-1}\left( -\frac{1}{2},\dots , -\frac{1}{2}\right)/\mathrm{U}(1)^{N-1},
\]
the reduction (and hence the space of FUNTFs) is connected if and only if the level set $\mu_{\mathrm{U}(1)^{N-1}}^{-1}\left( -\frac{1}{2},\dots , -\frac{1}{2}\right)$ is.

\subsection{The Frame Homotopy Conjecture}

We can now give a short proof of the frame homotopy conjecture, first proved by Cahill, Mixon and Strawn~\cite[Theorem 1.1]{cahill2017connectivity}.

\begin{thm}\label{thm:funtf_connected}
The space of complex FUNTFs is path-connected. 
\end{thm}

\begin{proof}
	Since the Grassmannian is compact and connected, Atiyah's Theorem (Theorem \ref{thm:atiyah}) applies and implies that $\mu_{\mathrm{U}(1)^{N-1}}^{-1}\left( -\frac{1}{2},\dots , -\frac{1}{2}\right) \subset \Gr$ is connected, which we have just seen implies the space of FUNTFs is connected. Moreover, the FUNTF space $\mathcal{F}_{\frac{N}{k}}(\vec{1})$ is a real algebraic set in $\C^{k \times N} \approx \R^{2\cdot k \cdot N}$ and is therefore locally path-connected (in fact, it is triangulable by \L{}ojasiewicz's Triangulation Theorem \cite{lojasiewicz1964triangulation}) so that connectivity implies path connectivity, completing the proof.
\end{proof}

\section{The General Case}\label{sec:general_case}

In this section we extend the proof strategy of Theorem~\ref{thm:funtf_connected} to treat general spaces of frames with prescribed frame operator and vector of squared norms and thereby prove the Main Theorem. The setup is essentially the same: given a Hermitian $k \times k$ matrix $S$ and a vector $\vec{r} = (r_1, \dots , r_{N-1})$ with $r_i \geq 0$, the space $\mathcal{F}_S(\vec{r})$ of frames $F$ with frame operator $FF^\ast = S$ and squared frame norms $\|f_i\|^2=r_i$ arises as a fiber of the momentum map $\mu$ for the $\mathrm{U}(k) \times \mathrm{U}(1)^{N-1}$ action on $\C^{k \times N}$:
\[
	\mathcal{F}_S(\vec{r}) = \mu^{-1}\left(S, -\frac{1}{2}\vec{r}\right).
\]

The challenge is that $\left(S, -\frac{1}{2}\vec{r}\right) \in \mathcal{H}(k) \times \R^{N-1}$ is not fixed by the coadjoint action of $\mathrm{U}(k) \times \mathrm{U}(1)^{N-1}$. The basic issue is that, for $U \in \mathrm{U}(k)$ and $F \in \mathcal{F}_S(\vec{r})$,
\[
	\mu(UF) = \left(UFF^\ast U^\ast,\left(-\frac{1}{2}\|Uf_1\|^2,\dots , -\frac{1}{2}\|Uf_{N-1}\|^2\right)\right) = \left(U S U^\ast,-\frac{1}{2}\vec{r}\right)
\]
and $USU^\ast \neq S$ in general. This means that $\mathrm{U}(k) \times \mathrm{U}(1)^{N-1}$ does not act on $\mathcal{F}_S(\vec{r})$ and hence there is no quotient.

However, there is still a natural symplectic reduction at hand when $\left(S, -\frac{1}{2}\vec{r}\right)$ is a regular value, this time the more general reduction over a coadjoint orbit
\[
	\C^{k \times N}\sslash_{\mathcal{O}_{\left(S,-\frac{1}{2}\vec{r}\right)}} \mathrm{U}(k) \times \mathrm{U}(1)^{N-1} := \mu^{-1}\left(\mathcal{O}_{\left(S,-\frac{1}{2}\vec{r}\right)}\right)/\left(\mathrm{U}(k) \times \mathrm{U}(1)^{N-1} \right),
\]
where $\mathcal{O}_{\left(S,-\frac{1}{2}\vec{r}\right)}$ is the coadjoint orbit of $\left(S,-\frac{1}{2}\vec{r}\right) \in \mathcal{H}(k) \times \R^{N-1}$. Since the coadjoint action of the torus $\mathrm{U}(1)^{N-1}$ on $\R^{N-1}$ is trivial, this is just 
\[
	\mathcal{O}_{\left(S,-\frac{1}{2}\vec{r}\right)} = \mathcal{O}_S \times \left\{-\frac{1}{2}\vec{r} \right\},
\]
where $\mathcal{O}_S = \{USU^\ast : U \in \mathrm{U}(k)\}$ is the orbit of $S \in \mathcal{H}(k)$ under the conjugation action of $\mathrm{U}(k)$ on $\mathcal{H}(k)$. Hence,
\[
	\widetilde{\mathcal F}_S(\vec{r}) := \mu^{-1}\left(\mathcal{O}_{\left(S,-\frac{1}{2}\vec{r}\right)}\right) = \{F \in \C^{k \times N} : FF^\ast = USU^\ast \text{ for some } U \in \mathrm{U}(k) \text{ and } \|f_i\|^2=r_i \text{ for all } i\}
\]
is the set of all frames with frame operator conjugate to $S$ and squared frame norms given by the $r_i$. 

In particular, performing reduction in stages yields
\[
	\widetilde{\mathcal F}_S(\vec{r})/\left( \mathrm{U}(k) \times \mathrm{U}(1)^{N-1}\right) = \C^{k \times N}\sslash_{\mathcal{O}_{\left(S,-\frac{1}{2}\vec{r}\right)}} \mathrm{U}(k) \times \mathrm{U}(1)^{N-1} = \left( \C^{k \times N} \sslash_{\mathcal{O}_S} \mathrm{U}(k)\right)\sslash_{-\frac{1}{2}\vec{r}} \mathrm{U}(1)^{N-1}.
\]
The inner reduction will turn out to be a generalization of a Grassmannian called a flag manifold, which is compact and connected, so connectedness of the full reduction -- and hence of $\widetilde{\mathcal F}_S(\vec{r})$ -- will again follow from Atiyah's connectedness theorem. 

The goals of this section, then, are: (i) to identify the regular values of $\mu$; (ii) to identify $\C^{k \times N} \sslash_{\mathcal{O}_S} \mathrm{U}(k)$ as a flag manifold and to observe that it is compact and connected; and (iii) to see that connectedness of $\widetilde{\mathcal F}_S(\vec{r})$ implies $\mathcal{F}_S(\vec{r})$ is connected.

\subsection{Regular Values of $\mu$}\label{sec:regular_values}
Since $\mu:\C^{k \times N} \to \mathcal{H}(k) \times \R^{N-1}$ is a product map, its regular values will be those points $(S,\vec{\xi}) \in \mathcal{H}(k) \times \R^{N-1}$ so that $S$ is a regular value of $\mu_{\mathrm{U}(k)}$ and $\vec{\xi}$ is a regular value of $\mu_{\mathrm{U}(1)^{N-1}}$.

In turn, $\mu_{\mathrm{U}(1)^{N-1}}: \C^{k \times N} \to \R^{N-1}$ is itself a product map, so understanding its regular values boils down to understanding the regular values of the map $\mu_{\mathrm{U}(1)}: \C^k \to \mathfrak{u}(1)^\ast \approx \R$ considered in the proof of Proposition~\ref{prop:torus moment}. The image of this map is the nonpositive reals, so zero is clearly a critical value of this map, and any $\xi < 0$ is regular since, for any $f \in \C^k$ and $X \in T_f \C^k \approx \C^k$,
\[
	D_f \mu_{\mathrm{U}(1)}(X) = -\langle f,X\rangle.
\]
Thus, $D_f \mu_{\mathrm{U}(1)}$ is the zero map if and only if $f =\vec{0}$, which happens if and only if $\mu_{\mathrm{U}(1)}(f)=0$.

On the other hand, consider the map $\mu_{\mathrm{U}(k)}:\C^{k \times N} \to \mathcal{H}(k)$.

\begin{lem}\label{lem:surjective}
For $F \in \C^{k \times N}$, the map $D_F\mu_{\mathrm{U}(k)}:T_F \C^{k \times N} \approx \C^{k \times N} \rightarrow T_{FF^\ast}\mathcal{H}(k) \approx \mathcal{H}(k)$ is surjective if and only if $F$ is full rank; equivalently, if and only if $F$ is a frame.
\end{lem}

\begin{proof}
If $F$ is not full rank, then there is a nonzero vector $v \in \mathrm{ker}(F^\ast)$. Then for any $X \in \C^{k \times N}$,
$$
v^\ast D_F\mu_{\mathrm{U}(k)}(X) v = v^\ast F X^\ast v + v^\ast X F^\ast v = 0.
$$
Since $v^\ast \mathrm{Id}_k v = v^\ast v \neq 0$, we see that this implies $\mathrm{Id}_k \in \mathcal{H}(k)$ is not in the image of $D_F\mu_{\mathrm{U}(k)}$, which is therefore not surjective.

On the other hand, suppose that $F$ is full rank. To prove that $D_F\mu_{\mathrm{U}(k)}$ is onto, we choose an arbitrary $W \in \mathcal{H}(k)$ and need to show that the equation
\begin{equation}\label{eqn:deriviative_surjective}
D_F\mu_{\mathrm{U}(k)}(X) = F X^\ast + X F^\ast = W
\end{equation}
has a solution $X \in \C^{k \times N}$.  Since $F$ is full rank, it has a right inverse $F^{-1}_R$, and we take $X=\frac{1}{2}W \left(F^{-1}_R\right)^\ast$. Substituting this into the left hand side of \eqref{eqn:deriviative_surjective}, we obtain
$$
F \left(\frac{1}{2}W \left(F^{-1}_R\right)^\ast\right)^\ast + \frac{1}{2}W \left(F^{-1}_R\right)^\ast F^\ast = \frac{1}{2} F F^{-1}_R W^\ast + \frac{1}{2} W \left(F F^{-1}_R\right)^\ast =\frac{1}{2}\left(W^\ast + W \right) = W,
$$
where the last equality follows because $W$ is Hermitian.
\end{proof}

\begin{lem}\label{lem:regular_values}
The regular values of $\mu_{\mathrm{U}(k)}$ are exactly the invertible, positive-definite Hermitian matrices.
\end{lem}

\begin{proof}
We first note that any such matrix $S$ has a square root; i.e. a size $k \times k$ matrix $W$ with $W W^\ast = S$. We then construct a size $k \times N$ block matrix 
$$
F =  \begin{pmatrix} W & 0 \end{pmatrix},
$$
which satisfies $\mu_{\mathrm{U}(k)}(F)=F F^\ast = WW^\ast = S$. Thus, $\mu_{\mathrm{U}(k)}$ surjects onto the space of invertible, positive-definite Hermitian matrices. 

Now note that $FF^\ast = S$ is invertible if and only if $F$ is full rank. Indeed, it is clear that invertibility of $S$ implies that $F^\ast$ has trivial kernel. On the other hand, suppose that $F F^\ast$ has a nonzero vector $w$ in its kernel. Then
$$
0 = \left<FF^\ast w, w\right> = \left<F^\ast w, F^\ast w\right>
$$
implies that $w \in \mathrm{ker}(F^\ast)$, and $F$ is not full rank. Together with Lemma \ref{lem:surjective}, this completes the proof.
\end{proof}

Combining all the results of this section, we conclude:
\begin{prop}\label{prop:regular values}
	The regular values of $\mu: \C^{k \times N } \to \mathcal{H}(k) \times \R^{N-1}$ are exactly those $\left(S, \vec{\xi}\right)$ where $S$ is invertible and positive-definite and the entries of $\vec{\xi}$ are all negative.
\end{prop}

\subsection{A flag manifold}

For each invertible, positive-definite $S\in \mathcal{H}(k)$, we say that a vector $\vec{r} = (r_1, \dots , r_N)$ with $r_i > 0$ is \emph{admissible} if $\widetilde{\mathcal F}_S(\vec{r})$ is nonempty. The admissible vectors of squared norms are completely characterized by the following theorem.

\begin{thm}[Casazza and Leon~\cite{casazza2010}]\label{thm:admissible}
A vector $\vec{r}=(r_1,\ldots,r_N)$ is an admissible vector of squared norms for $\mathcal{F}_S$, where $S$ is a prescribed frame operator with  eigenvalues $\lambda_1 \geq \lambda_2 \geq \cdots \geq \lambda_k > 0$, if and only if
$$
\sum_{j=1}^N r_j = \sum_{j=1}^k \lambda_j \;\; \mbox{ and } \;\; \sum_{j=1}^\ell r_j \leq \sum_{j=1}^\ell \lambda_j \; \forall \; 1 \leq \ell \leq k.
$$
\end{thm}

This theorem follows immediately from the classical Schur--Horn theorem \cite{horn1954doubly,schur1923uber}. A modern exposition of the Schur--Horn theorem using ideas from symplectic geometry appears in \cite{knutson2000symplectic}.

Hence for invertible, positive-definite $S\in \mathcal{H}(k)$ and admissible $\vec{r}$, Proposition~\ref{prop:regular values} and Theorem~\ref{thm:marsden_weinstein} imply that
\[
	\mu_{\mathrm{U}(k)}^{-1}\left(\mathcal{O}_S\right)/\mathrm{U}(k) = \C^{k \times N}\sslash_{\mathcal{O}_S} \mathrm{U}(k)
\]
is a symplectic manifold.

This space plays the same role as the Grassmannian $\Gr$ did in the previous section. The key fact about the Grassmannian was that it was compact and connected, and so our goal now is to identify this space and see that it is also compact and connected.

For an invertible, positive-definite Hermitian matrix $S \in \mathcal{H}(k)$ with $\ell$ distinct (necessarily real) eigenvalues $\lambda_1 > \lambda_2 > \cdots > \lambda_\ell > 0$ with multiplicities $k_1,k_2,\ldots,k_\ell$ adding to $k$, there is a diffeomorphism 
$$
\mathcal{O}_S \approx \mathrm{U}(k)/\left(\mathrm{U}(k_1) \times \mathrm{U}(k_2) \times \cdots \times \mathrm{U}(k_\ell)\right)
$$
(see \cite[Section II.1.d]{audin2012torus}). In other words, if $d_1=k_1$, $d_2=k_1+k_2$, $d_3=k_1+k_2+k_3, \ldots, d_\ell=k_1+k_2 + \cdots + k_\ell = k$, then $\mathcal{O}_S$ is the \emph{flag manifold}
$$
\Fl_k(d_1,\ldots , d_\ell) = \{(P_1,\ldots,P_\ell) \mid \mbox{$P_j$ is a $d_j$-dimensional subspace of $\C^k$ of such that $P_j \subset P_{j+1}$}\}.
$$
It follows that
$$
\dim \mathcal{O}_S = k^2 - k_1^2 -k_2^2 - \cdots - k_\ell^2.
$$
For a regular value $S$ of the momentum map $\mu_{\mathrm{U}(k)}$, then, $\C^{k \times N}\sslash_{\mathcal{O}_S} \mathrm{U}(k)$ is a smooth manifold of dimension $2k(N-k)+\dim\mathcal{O}_S$. 

Let $S$ be a Hermitian matrix with eigenvalue multiplicities as above. Then $S$ is diagonalizable by unitary matrices, and it follows easily that $\C^{k \times N}\sslash_{\mathcal{O}_S} \mathrm{U}(k)$ only depends on these multiplicities. That is, if $S'$ is a Hermitian matrix with the same multiplicities, then $\C^{k \times N}\sslash_{\mathcal{O}_S} \mathrm{U}(k) \approx \C^{k \times N}\sslash_{\mathcal{O}_{S'}} \mathrm{U}(k)$. This observation suggests that $\C^{k \times N}\sslash_{\mathcal{O}_S} \mathrm{U}(k)$ is itself a flag manifold, which we now prove.

\begin{prop}\label{prop:flags}
	With $S$, $k_i$, and $d_i$ as above, $\C^{k \times N}\sslash_{\mathcal{O}_S} \mathrm{U}(k) \approx \Fl_N(d_1, \ldots , d_\ell, N)$.
\end{prop}

\begin{proof}
	Let $W$ be a square root of $S$, as in the proof of Lemma~\ref{lem:regular_values}, and let 
	$$
	F_0 = \begin{pmatrix} W & 0 \end{pmatrix} \in \C^{k \times N},
	$$
	 so that $F_0F_0^*=S$. Define $R = F_0^*F_0 \in \mathcal{H}(N)$ and let $\widetilde{\mathcal{O}}_R$ be its orbit under the coadjoint (conjugation) action of $\mathrm{U}(N)$ on $\mathcal{H}(N)$. 
	
	The matrix $R$ is rank $k$ and has the same nonzero eigenvalues as $S$. Since $S$ is invertible and hence does not have zero as an eigenvalue, the eigenvalue multiplicities of $R$ are $k_1, \ldots , k_\ell, N-k$, meaning that $\widetilde{\mathcal{O}}_R$ is a copy of the flag manifold $\Fl_N(d_1, \ldots , d_\ell, N)$.
	
	Now, for any $[F] \in \C^{k \times N}\sslash_{\mathcal{O}_S} \mathrm{U}(k)$, which we think of as the frames in $\widetilde{\mathcal F}_S$ which are unitarily equivalent to $F$, the Gramian $F^\ast F$ has the same spectrum as $R$, and so lies in $\widetilde{\mathcal{O}}_R$. Moreover, if $F_1, F_2$ represent the same class in $\C^{k \times N}\sslash_{\mathcal{O}_S} \mathrm{U}(k)$, then $F_2 = U F_1$ for some $U \in \mathrm{U}(k)$, and hence
	\[
		F_2^\ast F_2 = (U F_1)^\ast (UF_1) = F_1^\ast U^\ast U F_1 = F_1^\ast F_1.
	\]
	Therefore, $[F] \mapsto FF^*$ defines a smooth map $G:\C^{k \times N}\sslash_{\mathcal{O}_S} \mathrm{U}(k) \to \widetilde{\mathcal{O}}_R \approx \Fl_N(d_1, \ldots , d_\ell, N)$, which we claim is a diffeomorphism.
	
	To see that this map is injective, suppose $G([F_1]) = G([F_2])$ for $[F_1],[F_2] \in \C^{k \times N}\sslash_{\mathcal{O}_S} \mathrm{U}(k)$; i.e., $F_1^\ast F_1 = F_2^\ast F_2$. This implies $F_1$ and $F_2$ have the same right singular vectors as well as the same singular values. Also, since $F_1F_1^\ast$ and $F_2F_2^\ast$ are both conjugate to $S$, and hence to each other, by unitary matrices, the left singular vectors of $V_1$ and $V_2$ are related by a unitary transformation. But then simply writing out the singular value decompositions of $F_1$ and $F_2$ shows that $F_2 = U F_1$ for some $U \in U(k)$, and hence $[F_1] = [F_2]$.
	
	On the other hand, to see that $G$ is surjective, suppose $P \in \widetilde{\mathcal{O}}_R$. Then 
	\[
		P=U^*R U = U^\ast F_0^\ast F_0 U= (F_0 U)^\ast (F_0 U)
	\]
	for some $U \in \mathrm{U}(n)$. But then $[F_0 U] \in \C^{k \times N}\sslash_{\mathcal{O}_S} \mathrm{U}(k)$ since
	\[
		(F_0 U)(F_0 U)^\ast = F_0 U U^\ast F_0^\ast = F_0 F_)^\ast = S,
	\]
	so $P = G([F_0 U])$.
	
	We've now shown that $G$ is bijective, and the inverse map $P \mapsto [F_0 U]$ is clearly smooth, so this completes the proof.
\end{proof}

Notice, in particular, that the Grassmannian $\Gr = \Fl(k,N)$, so this generalizes the construction in the FUNTF case. 

Moreover, just like the Grassmannian, the flag manifold $\C^{k \times N}\sslash_{\mathcal{O}_S} \mathrm{U}(k) \approx \Fl_N(d_1, \ldots , d_\ell, N)$ inherits a natural symplectic structure from $\C^{k \times N}$ and the $\mathrm{U}(1)^{N-1}$ action on $\C^{k \times N}$ descends to the flag manifold and we have
\[
	\Fl(d_1,\dots , d_\ell,N)\sslash_{-\frac{1}{2}\vec{r}} \mathrm{U}(1)^{N-1} \approx \C^{k \times N}\sslash_{\mathcal{O}_{\left(S,-\frac{1}{2}\vec{r}\right)}} \mathrm{U}(k) \times \mathrm{U}(1)^{N-1}  = \widetilde{\mathcal F}_S(\vec{r})/\left( \mathrm{U}(k) \times \mathrm{U}(1)^{N-1}\right).
\]

The flag manifold $\Fl_N(d_1, \ldots , d_\ell, N) = \mathrm{U}(N)/\left(\mathrm{U}(k_1) \times \mathrm{U}(k_2) \times \cdots \times \mathrm{U}(k_\ell) \times \mathrm{U}(N-k)\right)$ is clearly connected, since both Lie groups involved in its definition are, so 
\[
	\Fl(d_1,\dots , d_\ell,N)\sslash_{-\frac{1}{2}\vec{r}} \mathrm{U}(1)^{N-1}=\mu_{\mathrm{U}(1)^{N-1}}^{-1}\left(-\frac{1}{2}\vec{r}\right)/\mathrm{U}(1)^{N-1}
\]
is connected (and hence so is $\widetilde{\mathcal F}_S(\vec{r})$) if and only if the level set $\mu_{\mathrm{U}(1)^{N-1}}^{-1}\left(-\frac{1}{2}\vec{r}\right)$ is.

\subsection{Proof of the Main Theorem}

We are now ready to prove our main theorem.

\begin{maintheorem}
For any frame operator $S$ and for any admissible vector of squared norms $\vec{r}$, the space $\mathcal{F}_S(\vec{r})$ is path-connected.
\end{maintheorem}

\begin{proof}
Fix a frame operator $S$ and an admissible vector of squared norms $\vec{r}$. The flag manifold $\Fl(d_1,\dots , d_\ell,N)$ is compact and connected, so Atiyah's Theorem (Theorem~\ref{thm:atiyah}) implies $\mu_{\mathrm{U}(1)^{N-1}}^{-1}\left(-\frac{1}{2}\vec{r}\right) \subset \C^{k \times N}\sslash_{\mathcal{O}_S} \mathrm{U}(k)$ is connected. In turn, we have just seen this implies $\widetilde{\mathcal F}_S(\vec{r})$ is connected. By the same algebraic set argument used in the FUNTF case, $\widetilde{\mathcal F}_S(\vec{r})$ is also path-connected. 

For any points  $F_0,F_1 \in \mathcal{F}_S(\vec{r}) \subset \widetilde{\mathcal F}_S(\vec{r})$, there is a continuous path $\widetilde{F}_t:[0,1] \rightarrow \widetilde{\mathcal F}_S(\vec{r})$ joining them. There exists a continuous path $U_t:[0,1] \rightarrow \mathrm{U}(k)$ with $\widetilde{F}_t \widetilde{F}_t^\ast = U_t S U_t^\ast$ and $U_0 = \mathrm{Id}_k = U_1$, so we can amend our original path to obtain $F_t:[0,1] \rightarrow \mathcal{F}_S(\vec{r})$ via the formula $F_t = U_t^\ast \widetilde{F}_t$, which ensures $F_t F_t^\ast = S$ for all~$t$. Since $U_t$ is unitary, this alteration of the path fixes the column norms for all $t$, so that $F_t$ is a path in $\mathcal{F}_S(\vec{r}) \subset \mathcal{F}_S$ connecting $F_0$ and $F_1$ and the theorem follows.
\end{proof}

\section{Discussion}\label{sec:discussion}

The symplectic approach to thinking about frames that we have introduced in this paper should be much more broadly applicable. There are several promising directions for further applications of symplectic ideas to important problems in frame theory. For example, since we have seen that the space of FUNTFs appears as a level set of the momentum map corresponding to the Hamiltonian $U(k) \times U(1)^{N-1}$ action on $\C^{k \times N}$, flowing along the negative gradient directions of the squared norm of the moment map~\cite{Lerman:2005ue} becomes a viable means of ``fixing up'' a frame which is nearly a FUNTF. This gives a new approach to attacking the Paulsen problem~\cite{Bodmann:2010ib,Casazza:2013bn2,Kwok:2018vd,Hamilton:2018ta} which we intend to take up in a future paper; see~\cite{Oberwolfach} for a brief introduction to the key ideas.

Symplectic geometry should also be relevant to the phase retrieval problem~\cite{Balan:2006bc,Candes:2013ka,Bandeira:2014it,Conca:2015kq,Vinzant:2015gu}, which can be cast in the following way: an unknown signal vector $v \in \C^k$ is mapped to $\C^N$ by the analysis operator from~\eqref{eqn:analysis} and then the result is fed into the momentum map~\eqref{eqn:torus moment map} of the Hamiltonian $U(1)^N$ action on $\C^N \approx \C^{1 \times N}$, and the problem is to invert this composite map up to a global phase ambiguity.  

A key feature of compact symplectic manifolds is that they have a natural probability measure (the \emph{Liouville measure}) induced by the symplectic volume form. In particular, in the case of K\"ahler manifolds, which all of the spaces under discussion are, the symplectic volume form and Riemannian volume form agree, so Liouville measure is the Riemannian measure. However, this measure is often much more accessible to symplectic than to Riemannian techniques. For example, (an open, dense subset of) the quotient ${ \mathcal{F}_{\frac{N}{k}\mathrm{Id}_k}(\vec{1})/(\mathrm{U}(k) \times \mathrm{U}(1)^{N-1}) \approx \mathrm{Gr}_k(\C^N)\sslash_{\left(-\frac{1}{2},\dots,-\frac{1}{2}\right)} \mathrm{U}(1)^{N-1}}$ of the space $\mathcal{F}_{\frac{N}{k}\mathrm{Id}_k}(\vec{1})$ of FUNTFs in $\C^k$ is toric~\cite{Flaschka:2005dq}, which means the Liouville measure has a simple structure~\cite{Cantarella:2016iy}. Building on work with equilateral polygons in $\R^3$~\cite{Cantarella:2016bt}, this leads to an explicit algorithm for sampling random FUNTFs in $\C^2$~\cite{ShonkwilerFrames}, but a similar algorithm should exist in all dimensions. Either experimentally or theoretically, it would be interesting to get estimates or bounds on the probability that FUNTFs have various nice properties.

Also, symplectic geometry can work well in infinite dimensions. The symplectic reduction operation was already applied to infinite-dimensional symplectic manifolds in Marsden and Weinstein's initial paper on the subject (see Examples 6 and 7 in \cite[Section 4]{marsden1974reduction}). Moreover, analogues of the Atiyah--Guillemin--Sternberg connectivity and convexity theorems have been extended to several infinite-dimensional settings \cite{harada2006connectivity, smith2014connectivity}. This suggests that the symplectic ideas developed here may also be relevant to the study of frames in infinite-dimensional Hilbert spaces. 

Finally, we caution that spaces of frames in $\R^k$ are generally not symplectic, though they should appear as Lagrangian submanifolds of the corresponding complex frame spaces. Hence, it is not obvious how to extend our main theorem or other symplectic arguments to real frame spaces. In fact, the direct translation of the statement of our main theorem cannot be true: it follows from work of Kapovich and Millson~\cite{kapovich1995moduli} that the space of tight frames in $\R^2$ with squared frame vector norms $(4,4,4,1,1,1)$ is not connected, so the correct statement of the real version of our theorem must necessarily be more complicated.

\section*{Acknowledgments} 
\label{sec:acknowledgments}

This project grew out of a conversation we had at the CMO--BIRS Workshop on the Geometry and Topology of Knotting and Entanglement in Proteins in November, 2017, and we would like to thank the organizers, BIRS, and the Casa Matem\'atica Oaxaca for a very stimulating workshop. Conversations we had with virtually all the participants at the Oberwolfach Mini-Workshop on Algebraic, Geometric, and Combinatorial Methods in Frame Theory in October, 2018 significantly refined our thinking and opened our eyes to broader applications of our symplectic ideas, so we would also like to thank the organizers and all the participants as well as the Mathematisches Forschungsinstitut Oberwolfach. We are very grateful for ongoing conversations about frames with various friends and colleagues, especially Jason Cantarella, Martin Ehler, Simon Foucart, Milena Hering, Emily King, Chris Manon, Dustin Mixon, Louis Scharf, and Nate Strawn. This work was supported by a grant from the Simons Foundation (\#354225, Clayton Shonkwiler).


\bibliographystyle{plain}

\bibliography{needham_bibliography}

\end{document}